\documentclass[12pt]{amsart}

\usepackage[dvipdfmx]{graphicx}
\usepackage{graphicx}
\usepackage{verbatim}
\usepackage{amssymb}
\usepackage{amsbsy}
\usepackage{amscd}
\usepackage{amsmath}
\usepackage{amsthm}
\usepackage[mathscr]{eucal}
\usepackage{color}
\usepackage{hyperref}

\usepackage{geometry}
\newtheorem{theorem}{Theorem}[section]
\newtheorem{fact}{Fact}
\newtheorem{definition}{Definition}[section]
\newtheorem{lemma}{Lemma}[section]
\newtheorem{proposition}{Proposition}[section]
\newtheorem{corollary}{Corollary}[section]

\theoremstyle{definition}
\newtheorem{example}{Example}[section]
\newtheorem*{remark}{Remark}

\title[Semi-discrete linear Weingarten surfaces with singularities]{Semi-discrete linear Weingarten surfaces with Weierstrass-type representations and their singularities}

\author[M. Yasumoto]{Masashi Yasumoto}
\address[M. Yasumoto]{Osaka City University Advanced Mathematical Institute, 3-3-138 Sugimoto, Sumiyoshi-ku Osaka 558-8585, Japan}
\email[M. Yasumoto]{yasumoto@sci.osaka-cu.ac.jp}

\author[W. Rossman]{Wayne Rossman}
\address[W. Rossman]{Department of mathematics, Faculty of science, Kobe University, Rokkodai-cho 1-1, Nada-ku, Kobe, 657-8501, Japan}
\email[W. Rossman]{wayne@math.kobe-u.ac.jp}

\thanks{The first author was supported by the JSPS Program for Advancing Strategic International Networks to Accelerate the Circulation of Talented Researchers ``Mathematical Science of Symmetry, Topology and Moduli, Evolution of International Research Network based on OCAMI'' (PI: Y. Ohnita). The second author was partly supported by the Grant-in-Aid for Scientific Research (C) 15K04845 (PI: W. Rossman), and (S) 17H06127 (PI: M.-H. Saito). }



\begin{document}

\begin{abstract}
We establish what semi-discrete linear Weingarten surfaces with Weierstrass-type representations in $3$-dimensional Riemannian and Lorentzian spaceforms are, confirming their required properties regarding curvatures and parallel surfaces, and then classify them. We then define and analyze their singularities. In particular, we discuss singularities of (1) semi-discrete surfaces with non-zero constant Gaussian curvature, (2) parallel surfaces of semi-discrete minimal and maximal surfaces, and (3) semi-discrete constant mean curvature $1$ surfaces in de Sitter 3-space. We include comparisons with different previously known definitions of such singularities.  
\end{abstract}

\maketitle

\section{Introduction}\label{semis-sec1}

Smooth (spacelike) 
linear Weingarten surfaces in 
$3$-dimensional Riemannian or Lorentzian 
spaceforms are those for 
which the Gaussian and mean curvatures $K$ 
and $H$ satisfy an affine linear relation 
\[ \alpha K + 2 \beta H + \gamma = 0 \]
for constants $\alpha$, $\beta$ and 
$\gamma$ not all zero, and generally these 
surfaces will have singularities.  There are 
special cases of these surfaces that admit 
Weierstrass-type representations: 
\begin{enumerate}
\item minimal surfaces in $3$-dimensional 
Euclidean space $\mathbb{R}^3$ 
and their parallel surfaces, 
\item minimal surfaces in $3$-dimensional 
Minkowski space $\mathbb{R}^{2,1}$ 
and their parallel surfaces, 
\item surfaces in $3$-dimensional 
hyperbolic space $\mathbb{H}^3$ such that 
$\alpha = 1 - \beta$ and $\gamma = 
-1-\beta$, referred to here as 
linear Weingarten surfaces of 
Bryant type, or BrLW 
surfaces for short (note that 
flat surfaces occur when $\beta=0$), 
\item surfaces in $3$-dimensional 
de Sitter space $\mathbb{S}^{2,1}$ such 
that $\alpha = -1 - \beta$ and $\gamma = 
1-\beta$, referred to here as 
linear Weingarten surfaces of Bianchi 
type, or BiLW surfaces for short.  
\end{enumerate}

The case of fully discrete surfaces with 
Weierstrass-type representations was 
considered in \cite{RY2}, and the semi-discrete case is 
considered here.  Amongst our results, 
we establish the next 
two facts (see Section \ref{semis-sec4}), 
which are important for 
confirming that our choices for 
Weierstrass-type representations for 
semi-discrete surfaces are correct.  

\begin{fact}\label{semis-fact1}
Semi-discrete surfaces with 
Weierstrass representations 
satisfy the same affine linear relations between the Gaussian 
and mean curvatures as both smooth and fully discrete surfaces 
with Weierstrass representations do.
\end{fact}

In the smooth case, as mentioned in \cite{KU}, parallel surfaces of BrLW surfaces in $\mathbb{H}^3$, resp. BiLW surfaces in $\mathbb{S}^{2,1}$, are also BrLW surfaces, resp. BiLW surfaces, and these surfaces are classified into three types. Including minimal and maximal surfaces, there are then five types, like as listed in Fact \ref{semis-fact2} below. The same is true of fully discrete surfaces with Weierstrass-type representations, see \cite{RY2}.

In this paper we investigate semi-discrete linear Weingarten surfaces with Weierstrass-type representations. As will be seen later, together with explanations of the terminologies used, semi-discrete linear Weingarten surfaces are classified as in Fact \ref{semis-fact2} below.

\begin{fact}\label{semis-fact2}
Semi-discrete surfaces with Weierstrass-type representations can be classified into the following five types: 
\begin{enumerate}
\item minimal surfaces and their parallel surfaces in $\mathbb{R}^3$,
\item maximal surfaces and their parallel surfaces in $\mathbb{R}^{2,1}$,
\item flat surfaces in $\mathbb{H}^3$ 
and $\mathbb{S}^{2,1}$,
\item linear Weingarten surfaces of hyperbolic type in $\mathbb{H}^3$ 
and $\mathbb{S}^{2,1}$,
\item linear Weingarten surfaces of de Sitter type in $\mathbb{H}^3$ 
and $\mathbb{S}^{2,1}$.
\end{enumerate}
Parallel surfaces of each type belong again to the same type. 
\end{fact}

{\bf Singularities on semi-discrete surfaces.} In the smooth case, linear Weingarten surfaces with Weierstrass-type representations as listed in Fact \ref{semis-fact2} above might have singularitites. So it is natural to expect that semi-discrete linear Weingarten surfaces with Weierstrass-type representations also have some notion of ``singularities'' (for the fully discrete case, see \cite{HRSY}, \cite{RY2}, \cite{Yashi2014}). The main purpose in this paper is to clarify and characterize such singularities.

Let us remark on two previous works on singularities of semi-discrete surfaces:

\begin{enumerate}
\item In \cite{Yashi-next}, the first author described semi-discrete maximal surfaces in $\mathbb{R}^{2,1}$ and analyzed their singularities. Singularities of semi-discrete maximal surfaces were defined on the set of edges and so are called {\em singular edges}, and they reflect the property of non-spacelikeness of tangent planes of smooth maximal surfaces at singular points (see \cite{Yashi-next} and Definition \ref{semis-def55} here). 
\item Though singular edges can appear on semi-discrete spacelike surfaces with Weierstrass-type representations in Lorentzian spaceforms, they do not appear on such semi-discrete surfaces in Riemannian spaceforms. In \cite{Yashi-rev}, in order to consider singularities of general semi-discrete surfaces in Riemannian spaceforms as well, points along the smooth curves of the semi-discrete surfaces that could be singular were introduced, which were called {\em flat-parabolic-singular (FPS, for short) points}. FPS points are directly related to behaviors of the principal curvatures of semi-discrete surfaces. Applying this, singularities of particular semi-discrete surfaces were analyzed. 
\end{enumerate}

However, as already mentioned in \cite{Yashi-rev}, those FPS points did not identify certain  possible singularities that we would like to consider. So we need to modify the definition of FPS points of semi-discrete surfaces (see Definition \ref{semis-def51} here).  This enables us to analyze possible singularities that we could not analyze before. 

The semi-discrete case has some uniquely interesting singular behaviors, since it combines elements from both the smooth and fully discrete cases.  In the final section, we establish a definition of singularities on semi-discrete surfaces which takes into account that 
singularities can occur with respect to either the smooth parameter or the discrete parameter for the surface.  Because, like in the fully discrete case, this definition incorporates sign changes in the principal curvatures, we need to also include the possibilities of flat and parabolic points.  Examples of such singularities can be seen in Figure \ref{semis-fig1}.

We thus find ourselves in a situation where we have two independent notions of potential singularities of semi-discrete surfaces, one defined on edges and the other defined at points in the smooth curves of the surfaces. It is natural to look for relations between these two notions, and this is the purpose of Theorems \ref{semis-thm53} and \ref{semis-thm55} here. Specifically, we prove that singular points on semi-discrete maximal surfaces in $\mathbb{R}^{2,1}$ and semi-discrete CMC $1$ surfaces in $\mathbb{S}^{2,1}$ as defined in this paper imply existence of neighboring singular edges (see Theorems  \ref{semis-thm53} and \ref{semis-thm55}). Finally we give criteria for singular edges of semi-discrete CMC $1$ surfaces in $\mathbb{S}^{2,1}$ and prove the analogous result as in Theorem 1.2 in \cite{Yashi-next} for this case as well (see Theorem \ref{semis-thm54}). With these two theorems we see strong correspondence between the two notions of potential singularities, giving us further confidence in the usefulness of these two notions.

\begin{figure}[h] \label{semis-fig1}
\begin{center}
\includegraphics[width=150mm]{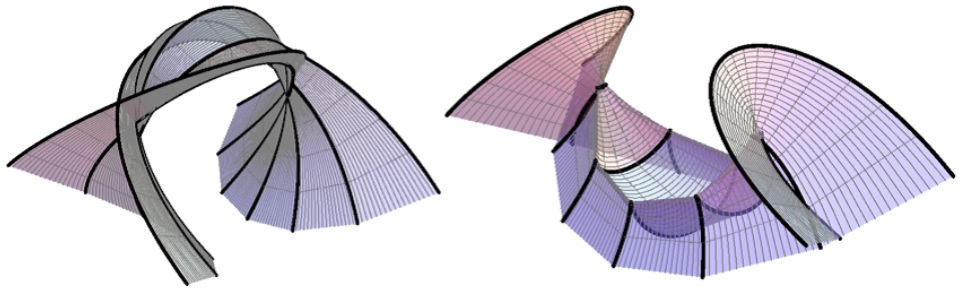}
\end{center}
\caption{Two different parallel surfaces of a 
semi-discrete Enneper minimal surface in 
$\mathbb{R}^3$, with singularities (which have the
appearance of cuspidal edges and swallowtails).}
\end{figure}

Along the way, we give criteria for determining singularities 
on parallel surfaces of semi-discrete minimal and maximal 
surfaces (see Theorem \ref{semis-thm51}), 
as well as on semi-discrete surfaces of 
Bryant and Bianchi types (see Theorem \ref{semis-thm52}).

\section{Semi-discrete Legendre immersions in $3$-dimensional spaceforms}\label{semis-sec2}

Let $M^3$ be a $3$-dimensional Riemannian or Lorentzian spaceform that is a 
quadric in a $4$-dimensional Riemannian or Lorentzian vector space $V^4$.  
A semi-discrete map is a map 
\[ x(k,t) : \mathbb{D} \rightarrow M^3 \; , \]
where $\mathbb{D}$ is a subdomain of $\mathbb{Z} \times \mathbb{R}$.  
We define derivatives and differences of $x$ by
\[
x=x(k,t), \ x_1 = x(k+1,t) , \ \partial x = \frac{dx}{dt} , \ \Delta x = x_1-x . 
\]
We will assume $x$ is a conjugate net, that is, $\partial x$, $\partial x_1$ and $\Delta x$ 
lie in a $2$-plane in $V^4$, called the 
{\em tangent plane} of the surface 
at the edge $[x,x_1]$ with endpoints 
$x$ and $x_1$.  

\begin{definition}
This map $x(k,t)$, together with a unit 
normal map $n(k,t)$, is called a 
{\em semi-discrete Legendre immersion} 
\[ \mathbb{D} \ni (k,t) \rightarrow 
(x,n) \in T_1 M^3 \] if it 
satisfies the following conditions:
\begin{enumerate}
\item $\partial n$, $\partial n_1$ and 
$\Delta n$ all lie in the tangent plane of 
the surface at the edge $[x,x_1]$, 
\item $\Delta x$, $n_1$ and $n$ all lie in one $2$-dimensional plane in $V^4$, 
\item $n$ is perpendicular to $\partial x$.  
\end{enumerate}
\end{definition}

Like for the fully discrete case, the curvature line condition in the discrete 
direction is partially built into condition (2) above, but we would additionally require that 
$\Delta n$ is parallel to $\Delta x$.  The curvature line condition in the 
smooth direction is simply that $\partial x$ and $\partial n$ are parallel, as in the next definition.  Existence of a curvature-line parametrization in the case of smooth surfaces rules out most types of umbilic points, and so in the following definition we are implicitly ruling out any semi-discrete surface with some notion of umbilic point.

\begin{definition}
If $\Delta n || \Delta x$ and 
$\partial x || \partial n$ 
and the tangent cross ratio satisfies
\[
cr(x,x_1):= \partial x \cdot (\Delta x)^{-1} \cdot \partial x_1 \cdot (\Delta x)^{-1} < 0 \; , 
\]
we say that $x$ is {\em curvature-line parametrized}.  
\end{definition}

To define the tangent cross 
ratio $cr(x,x_1)$ above 
requires that we multiply 
and invert points in $M^3$, which can be 
done as follows: we set \[ \mathbb{R}^3 := 
\{ (z_1,z_2,z_3,0) \, | \,  z_j 
\in \mathbb{R} \} \subseteq V = 
\mathbb{R}^4 := 
\{ (z_1,z_2,z_3,z_4) \, | \,  z_j 
\in \mathbb{R} \} \] with standard 
Euclidean metric 
\[
(z_1,z_2,z_3,z_4) \circ (w_1,w_2,w_3,w_4) = 
z_1w_1+z_2w_2+z_3w_3+z_4w_4 
\]
on $V$, and set $\mathbb{H}^3 = 
\mathbb{H}^3_+ \cup 
\mathbb{H}^3_-$, with 
\[ 
\mathbb{H}^3_+ := 
\{
(z_1,z_2,z_3,z_0) \, | \,  z_j 
\in \mathbb{R} , \, z_1^2 + z_2^2 + 
z_3^2 - z_0^2 = -1 , \, z_0 > 0 
\} \; , \]\[ 
\mathbb{H}^3_- := 
\{
(z_1,z_2,z_3,z_0) \, | \,  z_j 
\in \mathbb{R} , \, z_1^2 + z_2^2 + 
z_3^2 - z_0^2 = -1 , \, z_0 < 0 \} ,
\]
and 
\[
\mathbb{S}^{2,1} = \{
(z_1,z_2,z_3,z_0) \, | \,  z_j 
\in \mathbb{R} , \, z_1^2 + z_2^2 + 
z_3^2 - z_0^2 = 1  
\} ,
\]
all lying in 
\[
V=\mathbb{R}^{3,1} = 
\{ 
(z_1,z_2,z_3,z_0) 
\, | \, z_j \in \mathbb{R} \} 
\]
with the Minkowski metric 
\[
(z_1,z_2,z_3,z_0) \circ (w_1,w_2,w_3,w_0) = 
z_1w_1+z_2w_2+z_3w_3-z_0w_0 \; . 
\]
The relevant $4$-dimensional 
spaces are only $\mathbb{R}^4$ and 
$\mathbb{R}^{3,1}$, and we can identify 
points in those two spaces with 2 by 2 
matrices as follows:
\[ 
\mathbb{R}^4 \ni (z_1,z_2,z_3,z_4) 
\mapsto \begin{pmatrix}
z_1+i z_2 & z_3+i z_4 \\
-z_3+i z_4 & z_1-i z_2
\end{pmatrix} , \ \mathbb{R}^{3,1} \ni (z_1,z_2,z_3,z_0) 
\mapsto \begin{pmatrix}
z_0+z_3 & z_1-i z_2 \\
z_1+i z_2 & z_0-z_3
\end{pmatrix} .
\]
We then can regard multiplication and 
inversion of points in $M^3$ as 
multiplication and inversion of matrices. For example, in the case of $\mathbb{R}^{3,1}$,
\begin{eqnarray*}
&&(z_1,z_2,z_3,z_0) \circ (w_1,w_2,w_3,w_0)=\\
&&\frac{1}{2} \mathrm{tr} \left(
\begin{pmatrix} z_0+z_3 & z_1-i z_2 \\ z_1+i z_2 & z_0-z_3 \end{pmatrix} 
\begin{pmatrix} 0 & 1 \\ -1 & 0 \end{pmatrix}
\begin{pmatrix} w_0+w_3 & w_1-i w_2 \\ w_1+i w_2 & w_0-w_3 \end{pmatrix}^t 
\begin{pmatrix} 0 & 1 \\ -1 & 0 \end{pmatrix} \right) \ .
\end{eqnarray*}

Note that the tangent cross ratio being 
real, which means it is a real scalar 
multiple of the 2 by 2 identity matrix 
and then we can regard that scalar 
multiple as the tangent cross ratio 
itself, implies the circularity condition, 
that is, there is a circle through $x$ and $x_1$ which is tangent to $\partial 
x$ at $x$ and $\partial x_1$ at $x_1$.  

We can then define 
semi-discrete isothermic surface as 
follows: 

\begin{definition}
A semi-discrete $x$ in $M^3$ 
is {\em semi-discrete isothermic} if the equation 
\[
cr(x(k,t),x(k+1,t)) = \frac{\tau (t)}{\sigma (k)} < 0 
\]
holds, where $\tau=\tau(t) \in \mathbb{R}$ 
depends only on $t$ and $\sigma=\sigma(k) \in \mathbb{R}$ 
depends only on $k$.  
\end{definition}

\section{Curvatures of curvature-line parametrized semi-discrete surfaces}\label{semis-sec3}

First we define the principal curvatures: 

\begin{definition}\label{semis-def31}
For a semi-discrete Legendre map $(x,n)$, the scalar functions 
$\kappa_k(t)$, $\kappa_{k,k+1}(t)$ given by 
\[
\partial n = - \kappa_k(t) \partial x \; , \;\;\; 
\Delta n = - \kappa_{k,k+1}(t) \Delta x \; , \;\;\; 
\]
are called the {\em principal curvatures} of $x$.  Here we abbreviate 
\[
\kappa = \kappa_k(t) , \quad 
\kappa_1 = \kappa_{k+1}(t)  , \quad  
\kappa_{01} = \kappa_{k,k+1}(t) \quad (\text{also }  \kappa_{-10}=\kappa_{k-1,k}(t)).
\]
\end{definition}

The following Definition \ref{semis-def32} 
in the case of $M^3=\mathbb{R}^3$ was 
given in \cite{KW}, then in 
$M^3=\mathbb{R}^{2,1}$ in 
\cite{Yashi-next}.  The definition 
of $H$ for semi-discrete surfaces in 
general $3$-dimensional spaceforms 
$M^3$ was given in \cite{BHMR}, and 
here we also give the definition of $K$ 
for general $M^3$.  
For this definition we use the mixed area formulation found in \cite{BHMR}.

For two semi-discrete conjugate surfaces $x,y : \mathbb{D} \rightarrow V^4$ 
satisfying parallelity conditions $\partial x \parallel \partial y$ and $\Delta x 
\parallel \Delta y$, we define the mixed area element 
\[
A(x,y) := \frac{1}{4} ((\partial x+\partial x_1) \wedge \Delta y + 
(\partial y+\partial y_1) \wedge \Delta x) 
\; , \] where the operator $\wedge$ 
is defined by 
\[ 
(a \wedge b) c := (a \circ c) b - 
(b \circ c) a . 
\]

\begin{definition}\label{semis-def32}
Let $(x,n) : \mathbb{Z} \times \mathbb{R} 
\to T_1M^3$ be a 
semi-discrete curvature-line parametrized surface.  Then the Gaussian 
curvature $K$ and mean curvature $H$ of $x$ are 
defined, as functions on the set of edges $[x,x_1]$, by 
\[
A(n,n) = K \cdot A(x,x) \; , \;\;\; 
A(x,n) = - H \cdot A(x,x) \; . 
\]
Similarly , the Gaussian and mean curvatures of $n$ can be defined, regarding $x$ as the normal vector of $n$.
\end{definition}

Similarly to the arguments in \cite{BS} and 
\cite{Yashi-rev}, we have the following 
proposition: 

\begin{proposition} \label{semis-prop31}
Let $x$ be a semi-discrete curvature-line parametrized surface with Gauss map 
$n$ so that $(x,n)$ is 
a Legendre immersion.  Let $\kappa$, 
$\kappa_1$, $\kappa_{01}$, $K$, $H$ be 
the resulting principal, 
Gaussian and mean curvatures.  Then 
\[
K = \frac{\kappa_{01} (2 \kappa 
\kappa_1 - \kappa \kappa_{01}- \kappa_1\kappa_{01})}{\kappa_1+\kappa-2 \kappa_{01}} \; , \;\;\; 
H = \frac{\kappa \kappa_1 - \kappa_{01}^2}{\kappa_1+\kappa-2 \kappa_{01}} \; . 
\]
\end{proposition}

\begin{example}
Like as seen in \cite{RY1}, where semi-discrete catenoids in $\mathbb{R}^3$ 
with smooth profile curves were shown to have the same profile curves 
as smooth catenoids, one can now check here that, more generally, semi-discrete 
linear Weingarten surfaces in spaceforms with smooth profile curves 
have the same profile curves as their smooth counterparts.  
\end{example}

\section{Semi-discrete surfaces with Weierstrass representations}\label{semis-sec4}

\subsection{The cases of $\mathbb{R}^3$ and $\mathbb{R}^{2,1}$}\label{semis-subsec41}
Let $g$ be a semi-discrete holomorphic function, that is, a 
semi-discrete isothermic map into the plane, with tangent cross ratio 
factorizing functions $\tau$, $\sigma$.  We assume the semi-discrete 
analog of a smooth holomorphic function having a nonzero derivative, that is, 
$\partial g$ and $\Delta g$ are never zero, 
and we now state the Weierstrass-type 
representations for semi-discrete 
isothermic minimal 
and maximal surfaces, i.e. those with 
$H=0$ in $\mathbb{R}^3$ and 
$\mathbb{R}^{2,1}$: 

\begin{proposition}[\cite{RY1}, 
\cite{Yashi-next}] Any semi-discrete minimal (resp. 
maximal) surface in $\mathbb{R}^3$ (resp. 
$\mathbb{R}^{2,1}$) can be 
piecewise represented using a semi-discrete holomorphic function $g$ by solving
\begin{equation}\label{semis-eq1}
\partial x =  {\rm Re} \left( \frac{\tau}{2\partial g} \begin{pmatrix} 1-\epsilon g^2 \\ i (1+\epsilon g^2) \\ 2\epsilon g \end{pmatrix} \right) , \quad 
\Delta x = {\rm Re} \left( \frac{\sigma}{2\Delta g} \begin{pmatrix} 1-\epsilon g g_1 \\ i (1+\epsilon g g_1) \\ \epsilon 
(g+g_1) \end{pmatrix} \right) , 
\end{equation}
with $\epsilon = 1$ (resp. $\epsilon = -1$), 
and the normal field is 
\[
n =\frac{1}{1+\epsilon 
|g|^2} \begin{pmatrix} 2 
\epsilon \text{\em Re} (g) 
\\ 2 \epsilon \text{\em Im} (g) \\ 1-
\epsilon |g|^2 \end{pmatrix}. 
\]
\end{proposition}

Direct computation shows the following lemma.

\begin{lemma}\label{semis-lem41}
For any semi-discrete minimal (resp. 
maximal) surface, the $\kappa$ and 
$\kappa_{01}$ in Definition \ref{semis-def31}
satisfy
\begin{equation}\label{semis-eq2}
\kappa = \frac{-4 |\partial g|^2}{\tau 
(1+\epsilon |g|^2)^2} , \ \kappa_{01} = \frac{-4 |\Delta g|^2}{\sigma 
(1+\epsilon |g|^2)(1+\epsilon |g_1|^2)} . 
\end{equation}
\end{lemma}

One can also confirm this corollary: 

\begin{corollary}\label{semis-cor41}
For any choice 
of $\theta \in \mathbb{R}$, the 
parallel surface 
\[ 
x_\theta := x + \theta \cdot 
n 
\] 
satisfies the circularity condition, 
with Gaussian and mean curvatures 
\[ 
K_\theta^x 
=\frac{K_0}{1-2\theta \cdot H_0+\theta^2 \cdot K_0} 
 , \quad H_\theta^x 
=\frac{H_0-\theta 
K_0}{1-2\theta \cdot H_0+\theta^2 \cdot K_0} 
\]
satisfying 
\[ \frac{H_\theta^x}{K_\theta^x} = -\theta 
\; . \]
Also, the principal curvatures for 
$x_\theta$ satisfy 
\[ 
\kappa_\theta = \frac{\kappa}{1-\theta \cdot
\kappa}  , \quad 
\kappa_{01.\theta} = \frac{\kappa_{01}}{1-\theta \cdot
\kappa_{01}} \; . 
\]
\end{corollary}

\subsection{The cases of $\mathbb{H}^3$ and $\mathbb{S}^{2,1}$}
Taking the same $g$ as in Subsection 
\ref{semis-subsec41}, 
we make the genericity assumption 
\[
\mathcal{T}:=1+s g \overline{g} \neq 0 
\]
for some chosen 
constant $s \in \mathbb{R}$.  Take $\lambda \in 
\mathbb{R}$ to be any non-zero constant so that $1-\lambda 
\sigma \neq 0$.  
Solving, for $E \in \mathrm{GL}_2\mathbb{C}$, 
\begin{equation} \label{semis-eq3}
E^{-1} \Delta E = \begin{pmatrix}
0 & \Delta g \\ \frac{\lambda \sigma}{\Delta g} & 0 
\end{pmatrix} \ , \quad  
E^{-1} \partial E = \begin{pmatrix}
0 & \partial g \\ \frac{\lambda \tau}{\partial g} & 0 
\end{pmatrix} \ , 
\end{equation}
and defining 
\begin{eqnarray} \label{semis-eq4}
L = \begin{pmatrix}
0 & \sqrt{ \mathcal{T} } \\ \frac{-1}{\sqrt{ \mathcal{T} } } & 
\frac{-s \overline{g}}{ \sqrt{ \mathcal{T} } }
\end{pmatrix} \ , 
\end{eqnarray}
and the surface $x$ and its normal $n$ by
\begin{eqnarray} \label{semis-eq5}
 x = \frac{ \mathrm{sgn}  (\mathcal{T}) }{\det E} E L (\overline{E L})^t 
\ , \ 
n = \frac{ \mathrm{sgn}  (\mathcal{T}) }{\det E} E L \begin{pmatrix}
1 & 0 \\ 0 & -1 
\end{pmatrix}
(\overline{E L})^t \ ,
\end{eqnarray}
we will see that these are discrete BrLW surfaces and BiLW surfaces in $\mathbb{H}^3$ and 
$\mathbb{S}^{2,1}$, respectively. First, analogous to the discrete case, we have the following proposition.

\begin{proposition} \label{semis-prop42}
Semi-discrete BrLW surfaces in $\mathbb{H}^3$ and BiLW surfaces in $\mathbb{S}^{2,1}$ with Weierstrass-type representations as in Equations \eqref{semis-eq3}, \eqref{semis-eq4}, \eqref{semis-eq5} are circular nets.
\end{proposition}

\begin{proof}
Let $x$ be a BrLW surface in $\mathbb{H}^3$ described by a semi-discrete holomorphic function $g$. Observing that $\det E$ does not depend on the smooth parameter $t$, we have
\begin{eqnarray*}
&&\partial x =c_1 E \begin{pmatrix} 0 & \partial g (1+s |g|^2) \\ \partial \bar{g} (1+s |g|^2) & -s(\bar{g} \partial g+g \partial \bar{g}) \end{pmatrix} \overline{E}^t \ , \\
&&\partial x_1=c_2 E 
\begin{pmatrix}
1 & \Delta g \\
\frac{\lambda \sigma}{\Delta g} & 1
\end{pmatrix}
\begin{pmatrix} 
0 & \partial g_1 (1+s |g_1|^2) \\ \partial \bar{g_1} (1+s |g_1|^2) & -s(\bar{g_1} \partial g_1+g_1 \partial \bar{g_1}) 
\end{pmatrix}
\begin{pmatrix}
1 & \frac{\lambda \sigma}{\overline{\Delta g}} \\
\overline{\Delta g} & 1
\end{pmatrix}
\overline{E}^t , \\
&&\Delta x = c_3 E
\begin{pmatrix}
|\Delta g|^2 (1+s|g|^2) & \Delta g (1+s |g|^2) \\ 
\overline{ \Delta g }(1+s |g|^2) & \lambda \sigma (1+s |g_1|^2)-s(|g_1|^2 -|g|^2)
\end{pmatrix}
\overline{E}^t ,
\end{eqnarray*}
where
\begin{eqnarray*}
&&c_1 :=\frac{(1-s)| \partial g |^2+\lambda \tau (1+s |g|^2)^2 }{ | \partial g |^2 (1+s |g|^2)^2 \det E } , \quad c_2:=\frac{(1-s)| \partial g_1 |^2+\lambda \tau (1+s |g_1|^2)^2 }{ | \partial g_1 |^2 (1+s |g_1|^2)^2 (1-\lambda \sigma) \det E } , \\
&&c_3:=\frac{(1-s)|\Delta g|^2+\lambda \sigma (1+s|g|^2)(1+s|g_1|^2)}{ |\Delta g|^2 (1+s|g|^2)(1+s|g_1|^2) (1-\lambda \sigma )\det E } \ .
\end{eqnarray*}
By an isometry of $\mathbb{R}^{3,1}$, without loss of generality, we can assume that $E=
\begin{pmatrix}
1 & 0 \\ 0 & 1
\end{pmatrix}$ at one point.
Using the tangent cross ratio condition $\displaystyle cr(g,g_1)=\frac{\tau}{\sigma}$, by a calculation, we have
\[
\partial x \cdot (\Delta x)^{-1} \cdot \partial x_1 \cdot (\Delta x)^{-1} =cr (x,x_1) \begin{pmatrix}
1 & 0 \\ 0 & 1
\end{pmatrix}
\]
with
\[
cr (x,x_1)=\frac{\sigma (1-\lambda \sigma)}{\tau} \cdot \frac{ \{ (1-s)|\partial g|^2+\lambda \tau (1+s|g|^2)^2 \} \{ (1-s)|\partial g_1|^2+\lambda \tau (1+s|g_1|^2)^2 \} }{ \{ (1-s)|\Delta g|^2+\lambda \sigma (1+s|g|^2) (1+s|g_1|^2) \}^2 } \ .
\]
Thus $x$ is a circular net. Note that $x$ is not generically semi-discrete isothermic. 

A proof that $n$ is a semi-discrete circular net will be given just after Lemma \ref{semis-lem42}.
\end{proof}

Direct computations confirm this lemma:

\begin{lemma}\label{semis-lem42}
For any allowed choice of $s$, we have the following: 
\begin{itemize}
\item $\partial x \parallel \partial n$, $\Delta x \parallel \Delta n$ in 
$\mathbb{R}^{3,1}$, and the principal curvatures in Definition \ref{semis-def31} 
satisfy  
\begin{equation}\label{semis-eq6}
\kappa = \frac{|\partial g|^2 (-1-s)+
(1+s|g|^2)^2 \lambda \tau}{|\partial g|^2 (1-s)+ 
(1+s|g|^2)^2 \lambda \tau} \; , \ 
\kappa_{01} = \frac{|\Delta g|^2 (-1-s)+ 
(1+s|g|^2) (1+s|g_1|^2) \lambda \sigma}{|\Delta g|^2 (1-s)+ 
(1+s|g|^2) (1+s|g_1|^2) \lambda \sigma} \; . 
\end{equation}
%%%%%%%
%
\item $1+s |g|^2 > 0$, resp. $1+s |g|^2 < 0$, if and only if $x$ lies in $\mathbb{H}^3_+$, 
resp. $\mathbb{H}^3_-$.  
\item $\Delta x$, $\partial x$, $\partial x_1$ 
lie in a plane (that is generically 
spacelike) in $\mathbb{R}^{3,1}$, 
and thus $x$ satisfies the circularity condition.    
\end{itemize}
\end{lemma}

Now we show the semi-discrete circularity of semi-discrete BiLW surfaces in $\mathbb{S}^{2,1}$. Let $n$ be a semi-discrete BiLW surface described by a semi-discrete holomorphic function $g$. By Lemma \ref{semis-lem42}, we have 
\begin{eqnarray*}
&&\partial n \cdot (\Delta n)^{-1} \cdot \partial n_1 \cdot (\Delta n)^{-1} = \frac{\kappa \kappa_1}{\kappa_{01}^2}\partial x \cdot (\Delta x)^{-1} \cdot \partial x_1 \cdot (\Delta x)^{-1} \\
&&=\frac{\sigma (1-\lambda \sigma)}{\tau} \cdot \frac{ \{ (-1-s)|\partial g|^2+\lambda \tau (1+s|g|^2)^2 \} \{ (1-s)|\partial g_1|^2+\lambda \tau (1+s|g_1|^2)^2 \} }{ \{ (-1-s)|\Delta g|^2+\lambda \sigma (1+s|g|^2) (1+s|g_1|^2) \}^2 }
\begin{pmatrix}
1 & 0 \\ 0 & 1
\end{pmatrix} .
\end{eqnarray*}
Thus $n$ is also semi-discrete circular, proving the last part of Proposition \ref{semis-prop42}.

Furthermore, combining Proposition \ref{semis-prop42} and Lemma \ref{semis-lem42}, we can show the following curvature properties of semi-discrete BrLW and BiLW surfaces, which also imply Fact \ref{semis-fact1} in the introduction:

\begin{proposition} \label{semis-prop43}
A semi-discrete BrLW surface $x$ in $\mathbb{H}^3$ and a semi-discrete BiLW surface $n$ in $\mathbb{S}^{2,1}$ described by a semi-discrete holomorphic function $g$ via Equations \eqref{semis-eq3}, \eqref{semis-eq4}, \eqref{semis-eq5} satisfy the following curvature conditions:
\begin{equation}\label{semis-eq7}
2s(H^{x}-1)+(1-s)(K^{x}-1)=0, \quad 2s(H^{n}-1)-(1+s)(K^{n}-1)=0 ,
\end{equation}
where $H^{x}$ and $K^{x}$ are the mean and Gaussian curvatures of $x$ and $H^{n}$ and $K^{n}$ are the mean and Gaussian curvatures of $n$.
\end{proposition}

\begin{proof}
The curvature condition for $x$ can be obtained by a direct but tedious calculation, which we omit here. We now see the curvature condition for $n$ from the curvature condition for $x$: The surfaces 
satisfy $\displaystyle K^x=\frac{1}{K^n} , \ H^x=\frac{H^n}{K^n}$. Substituting $K^x, H^x$ into the curvature condition for $x$, we have the curvature condition for $n$, proving the proposition.
\end{proof}

Like in the smooth and fully discrete cases, we define types of semi-discrete BrLW and BiLW surfaces as follows:
\begin{definition}
The surfaces $x$ and $n$ are said to be of {\em hyperbolic type} if $s>0$, and of {\em de Sitter type} if $s<0$.  
\end{definition}

Let $x$ be a semi-discrete BrLW surface in $\mathbb{H}^3$ and let $n$ be a semi-discrete BiLW surface in $\mathbb{S}^{2,1}$ described by a single choice of $g$ and $s$. Then we define the parallel surface $x_{\theta}$ of $x$ at distance $\theta$ $(\theta \in \mathbb{R})$ as 
\[
x_{\theta}:=\cosh \theta \cdot x +\sinh \theta \cdot n \in \mathbb{H}^3.
\]
One can confirm the following proposition, which proves Fact \ref{semis-fact2} in the introduction:

\begin{proposition}\label{semis-prop44}
For any choices of $s$ and $\theta \in \mathbb{R}$, the 
parallel surface $x_\theta$ of a semi-discrete circular surface $x$ in $\mathbb{H}^3$ with unit normal vector field 
\[
n_\theta :=\sinh \theta \cdot x+\cosh \theta \cdot n \in \mathbb{S}^{2,1}
\]
satisfies the circularity condition, and 
\begin{eqnarray*}
K_\theta^x=\frac{K_0^x \cosh^2 \theta -H_0^x \sinh (2 \theta) +\sinh ^2 \theta }{\cosh^2 \theta -H_0^x \sinh (2 \theta) + 
K_0^x \sinh ^2 \theta } \ , \quad 
H_\theta^x = \frac{-(K_0^x+1)\sinh (2\theta)+2H_0^x 
\cosh (2\theta)}{2\{ 
\cosh^2 \theta -H_0^x \sinh (2\theta) + 
K_0^x \sinh^2 \theta \} } \ .
\end{eqnarray*} 
In particular, if $x$ is a semi-discrete BrLW surface in $\mathbb{H}^3$, $x_{\theta}$ is also of Bryant type satisfying
\begin{equation}\label{semis-eq8}
2s_\theta(H_\theta^x-1)+(1-s_\theta)(K_\theta^x-1)=0 \; ,
\end{equation}
where $s_\theta=\mathrm{e}^{-2\theta }s$, and $x_{\theta}$ can be also obtained from the Weierstrass-type representation.

Moreover, the normal $n_\theta$ also satisfies the circularity condition, and is of Bianchi type satisfying 
\begin{equation}\label{semis-eq9}
2s_\theta(H_\theta^x-1)-(1+s_\theta)(K_\theta^x-1)=0 ,
\end{equation}
and $n_{\theta}$ can be also obtained from the Weierstrass-type representation.
\end{proposition}

\begin{proof}
Combining Lemma \ref{semis-lem42} and the definition of the tangent cross ratio, we can easily show the circularity conditions for $x_{\theta}$ and $n_{\theta}$. Next we determine the mean and Gaussian curvatures of a parallel surface of a semi-discrete circular surface in $\mathbb{H}^3$. Let $K^x_0, H^x_0$ (resp. $K^x_{\theta}, H^x_{\theta}$) be the Gaussian and mean curvatures of $x$ (resp. $x_\theta$). By a calculation, we have
\begin{eqnarray*}
A(x_\theta , x_\theta)&=&\cosh ^2 \theta \cdot A(x,x)+\sinh (2\theta) \cdot A(x,n)+\sinh ^2 \theta \cdot A(n,n) \\
&=& \{ \cosh ^2 \theta -H^x_0 \sinh (2\theta)+ K^x_0 \sinh ^2 \theta \} \cdot A(x,x) .
\end{eqnarray*}
Similarly, we have
\begin{eqnarray*}
&&A(n_\theta , n_\theta)=\{ \sinh ^2 \theta -H^x_0 \sinh (2\theta)+K^x_0 \cosh ^2 \theta \} \cdot A(x,x), \\
&&A(x_\theta , n_\theta)=\frac{1}{2} \{ (K^x_0+1) \sinh (2\theta) -2H^x_0 \cosh (2 \theta) \} \cdot A(x,x) .
\end{eqnarray*}
Thus we have $K^x_{\theta}, H^x_{\theta}$ of the forms as in Proposition \ref{semis-prop44}.

Here we assume that $(x,n)$ is a pair of a semi-discrete BrLW surface in $\mathbb{H}^3$ and a semi-discrete BiLW surface in $\mathbb{S}^{2,1}$ described by a single choice of $g$ and $s$. Note that $x$ is a parallel surface of $x_\theta$ with distance $-\theta$, that is, $x=(x_\theta)_{-\theta}$. Then we have
\begin{eqnarray*}
K^x_0=\frac{K_\theta^x \cosh^2 \theta +H_\theta^x \sinh (2 \theta) +\sinh ^2 \theta }{\cosh^2 \theta +H_\theta^x \sinh (2 \theta) + 
K_\theta^x \sinh ^2 \theta } \ , \quad 
H_0^x = \frac{(K_\theta^x+1)\sinh (2\theta)+2H_\theta^x 
\cosh (2\theta)}{2\{ 
\cosh^2 \theta +H_\theta^x \sinh (2\theta) + 
K_\theta^x \sinh^2 \theta \} } \ .
\end{eqnarray*}
Substituting these into Equation \eqref{semis-eq7}, we have the relation \eqref{semis-eq8}. By a similar argument as in the proof of Proposition \ref{semis-prop43}, we have Equation \eqref{semis-eq9}.

Finally, we show that any parallel surfaces of semi-discrete BrLW and BiLW surfaces with Weierstrass-type representations can be also described by Weierstrass-type representations. First we consider parallel surfaces of semi-discrete BrLW surfaces in $\mathbb{H}^3$. Let $x$ be a semi-discrete BrLW surface in $\mathbb{H}^3$ and $x_\theta$ be a parallel surface of $x$ with distance $\theta$. Here we assume that $\mathcal{T}>0$ (even when $\mathcal{T}<0$, the conclusion is the same). Then
\begin{eqnarray*}
&&x_\theta =\frac{1}{\det E} EL \begin{pmatrix} e^{\theta} & 0 \\ 0 & e^{-\theta} \end{pmatrix} \overline{(EL)}^t = \frac{1}{\det E} E \begin{pmatrix} e^{-\theta} (1+s g \bar{g}) & -s e^{-\theta} g \\ -s e^{-\theta} \bar{g} & \frac{ e^{\theta}(1+s^2 e^{-2\theta} g \bar{g}) }{1+s g \bar{g} } \end{pmatrix} \overline{E}^t \\
&&=\frac{1}{\det E} \left( E \begin{pmatrix} e^{-\theta/2} & 0 \\ 0 & e^{\theta/2} \end{pmatrix} \right)
\begin{pmatrix} 1+s g \bar{g} & -s e^{-\theta} g \\ -s e^{-\theta} \bar{g} & \frac{ 1+s^2 e^{-2\theta} g \bar{g} }{1+s g \bar{g} } \end{pmatrix} \overline{\left( E \begin{pmatrix} e^{-\theta/2} & 0 \\ 0 & e^{\theta/2} \end{pmatrix} \right)}^t=\frac{1}{\det \tilde{E} } \tilde{E} \tilde{L} \overline{( \tilde{E} \tilde{L} )}^t \ ,
\end{eqnarray*}
where $\tilde{E}:=E \begin{pmatrix} e^{-\theta/2} & 0 \\ 0 & e^{\theta/2} \end{pmatrix}$ and $\tilde{L}:=\begin{pmatrix} 1+\tilde{s} \tilde{g} \overline{ \tilde{g} } & -\tilde{s} \tilde{g} \\ -\tilde{s}\overline{ \tilde{g} } & \frac{ -\tilde{s} \overline{ \tilde{g} } }{ 1+\tilde{s} \tilde{g} \overline{ \tilde{g} } }  \end{pmatrix} \ (\tilde{g}=e^{\theta} g, \tilde{s}=s e^{-2\theta})$. By the definition of $E$, $\tilde{E}$ is a solution of 
\[
\partial \tilde{E}= \tilde{E} \begin{pmatrix} 0 & \partial \tilde{g} \\ \frac{\lambda \tau}{ \partial \tilde{g} } & 0 \end{pmatrix}, \quad \Delta \tilde{E}= \tilde{E} \begin{pmatrix} 0 & \Delta \tilde{g} \\ \frac{\lambda \sigma}{ \Delta \tilde{g} } & 0 \end{pmatrix} .
\]
Thus $x_\theta$ can be obtained via the Weierstrass-type representation by replacing $g$ in Equations \eqref{semis-eq3}, \eqref{semis-eq4}, \eqref{semis-eq5} with $\tilde{g}$ and choosing $\tilde{s}$. Similarly, we consider semi-discrete BiLW surfaces in $\mathbb{S}^{2,1}$. Let $n$ be a semi-discrete BiLW surface in $\mathbb{S}^{2,1}$ and let $n_{\theta}$ be a parallel surface of $n$ at distance $\theta$. Then $n_{\theta}$ can be obtained via the Weierstrass-type representation by replacing $g$ in Equations \eqref{semis-eq3}, \eqref{semis-eq4}, \eqref{semis-eq5} with $\tilde{g}$ and choosing $\tilde{s}$, proving the proposition.
\end{proof}

We have thus arrived at Facts \ref{semis-fact1} and \ref{semis-fact2} in the introduction. 
The smooth and fully discrete cases can be found in \cite{GMM-second}, \cite{KU}, \cite{RY2}.

\section{Singularities of semi-discrete surfaces with Weierstrass representations}\label{semis-sec5}

Influenced by definitions of singularities 
in the smooth and fully discrete cases, we 
make the following definition, refining the 
definition in \cite{Yashi-rev} (recall that 
flat points on smooth surfaces are those for which both principal curvatures are zero, 
and parabolic points are those for which 
exactly one principal curvature is zero):  

\begin{figure}
\begin{center}
\includegraphics[width=150mm]{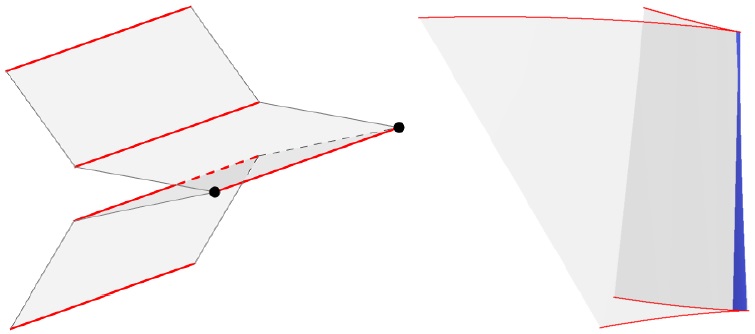}
\end{center}
\caption{Left: A typical example of the first part of item (1) in Definition \ref{semis-def51}. Right: A typical example of the first part of item (2).}
\label{semis-fig2}
\end{figure}

\begin{definition}\label{semis-def51}
We say that a point $(k_0,t_0)$, also its image $x(k_0,t_0)$, is 
a flat (F) or parabolic (P) or singular 
(S) point of the semi-discrete surface 
$x(k,t)$, with respect to either the discrete direction represented 
by changing $k$ (see the left-hand side of Figure \ref{semis-fig2}) or the smooth direction represented by changing $t$ (see the right-hand side of Figure \ref{semis-fig2}), as follows: 

\begin{enumerate}
\item $x(k_0,t_0)$ is an FPS point with respect to the discrete 
direction if 
\begin{eqnarray*} 
\kappa_{k_0-1,k_0}(t_0) \cdot 
\kappa_{k_0,k_0+1}(t_0) < 0, \text{ or at least one of } \kappa_{k_0-1,k_0}(t_0), 
\kappa_{k_0,k_0+1}(t_0) \text{ is 
infinite.} 
\end{eqnarray*}
\item $x(k_0,t_0)$ is an FPS point with respect to the smooth 
direction if 
\begin{eqnarray*}
\kappa_{k_0-1}(t_0) \cdot \kappa_{k_0}(t_0) < 0 \text{ or } 
\kappa_{k_0}(t_0) \cdot \kappa_{k_0+1}(t_0) < 0 
\text{ or } \\ 
\text{at least one of } \kappa_{k_0-1}(t_0), \kappa_{k_0}(t_0), \kappa_{k_0+1}(t_0) 
\text{ is infinite.} 
\end{eqnarray*}
\end{enumerate}
In the latter cases of either (1) or (2) above, where infinite values occur, we can say that $x(k_0,t_0)$ is a singular (S) point.  
\end{definition}

Like in \cite{RY2}, 
we are interested in cases where we can 
differentiate between FP and S vertices.  
This is the purpose of the next two 
definitions, which are independent of whether the surface has a 
Weierstrass representation.  

\begin{definition} \label{semis-def52}
We say that a semi-discrete circular
surface $x$ is {\em embedded} at a 
given edge $[x,x_1]$ if $\partial x$ and 
$\partial x_1$ lie to the same side of the 
line through $\Delta x$ within the tangent 
plane. Embeddedness of the Gauss map $n$ is similarly defined.
\end{definition}

Generically, in the (non-umbilic) smooth case, rank $1$ singularities of a surface correspond to flat or parabolic points of the surface's unit normal field, and vice versa. When $x$ or $n$ is not locally embedded, we certainly have a singular point, and this motivates the next definition.

\begin{definition} \label{semis-def53}
Let $x$ be a semi-discrete circular surface with bounded principal curvatures and spacelike tangent planes on its edges. 
Suppose that $x(k_0,t_0)$ is an FPS point with respect to just the smooth direction, and that precisely one of $\ell_{-10}:= \kappa_{k_0-1}(t_0) \cdot \kappa_{k_0}(t_0)$ and $\ell_{10} := \kappa_{k_0}(t_0) \cdot \kappa_{k_0+1}(t_0)$ is negative and the other is positive. Then 
we call $x(k_0,t_0)$ an FP point (i.e. non-singular), resp. a singular (S) point, if the Gauss map is not embedded, resp. is embedded, on the edge corresponding to the $\ell_{*0}$ that is negative at $t=t_0$ ($*=-1$ or $1$).  
\end{definition}

\begin{remark}
Suppose $\kappa_{k_0-1}(t_0) \cdot \kappa_{k_0}(t_0)$ is negative. 
Then $n$ is embedded on the edge 
$[x_{k_0-1},x_{k_0}]$ at $t=t_0$ if and only if 
$\text{cr}(n_{k_0-1},n_{k_0})<0$, since the 
tangent plane is spacelike.  By definition, 
$\text{cr}(x_{k_0-1},x_{k_0})=
\kappa^2_{k_0-1,1}
\kappa_{k_0-1}^{-1} \kappa_{k_0}^{-1} 
\text{cr}(n_{k_0-1},n_{k_0})$, so $n$ is 
embedded if and only if 
$\text{cr}(x_{k_0-1},x_{k_0})>0$.
\end{remark}

It was pointed out in \cite{RY2} that 
parallel surfaces of minimal and maximal 
surfaces never have flat or parabolic points, 
and this was used in that work 
to justify the analog 
of the definition below for the fully 
discrete surface case.  Similarly, the 
definition just below is also justified in 
the semi-discrete case.  In fact, we 
can see from the formulas for 
$\kappa_\theta$ and $\kappa_{01,\theta}$ 
in Corollary \ref{semis-cor41} that these 
principal curvatures are never zero.  

\begin{definition} \label{semis-def54}
On any parallel surface of a semi-discrete minimal or maximal 
surface, all FPS points are called simply {\em singular (S) points}.
\end{definition}

Lemma \ref{semis-lem42} now provides proofs of the following two theorems.

\begin{theorem}\label{semis-thm51}
For a parallel surface $x_\theta$ 
of a semi-discrete 
minimal or maximal 
surface at oriented distance 
$\theta$ (in the maximal case 
$\epsilon=-1$ we assume $|g_{-1}|$, 
$|g|$, $|g_1|$ 
are all not $1$), the condition for 
$\kappa_{-10} \cdot \kappa_{01}$ to be nonpositive -- that is, $x_{\theta}(k_0,t_0)$ is 
singular with respect to the discrete direction -- is 
\[
\theta \in [\min(a_{-1},a_1),
\max(a_{-1},a_1)] \; , 
\]
where 
\[ a_{*}= \frac{-\sigma (1+
\epsilon |g|^2)(1+\epsilon 
|g_{*}|^2)}{4 |\Delta g_{0*}|^2}
\; , \;\;\; * \in \{ -1,1 \} \; , \]
and the condition for $\kappa \cdot 
\kappa_1$ to be nonpositive -- that is, $x_{\theta}(k_0,t_0)$ is 
singular with respect to the smooth direction (with regard to the edge $[ x_{\theta}(k_0,t_0), x_{\theta}(k_0+1,t_0) ]$) -- is 
\[
\theta \in [\min(b,b_1),\max(b,b_1)] \; , 
\]
where 
\[ 
b= \frac{-\tau (1+ 
\epsilon |g|^2)^2}{4 |\partial g|^2} \; , \;\;\; 
b_1= \frac{-\tau (1+ 
\epsilon |g_1|^2)^2}{4 |\partial g_1|^2} \; .
\]
One can analogously give a condition for $x_{\theta} (k_0,t_0)$ to be singular in the smooth direction with regard to the edge $[ x_{\theta}(k_0-1,t_0), x_{\theta}(k_0,t_0) ]$.
\end{theorem}

For notational simplicity, set 
\[
\alpha_{-1} (s):= 1+s |g_{-1}|^2, \ \alpha (s):= 1+s |g|^2, \ \alpha_1 (s):= 1+s |g_1|^2. 
\]
We then have the next theorem.

\begin{theorem}\label{semis-thm52}
Let $x$ be a semi-discrete BrLW surface 
with Gauss map $n$ a semi-discrete BiLW surface.  
The condition for $\kappa^x_{-10} \cdot 
\kappa^x_{01}$ (equivalently, $\kappa^n_{-10} \cdot 
\kappa^n_{01}$) to be nonpositive -- that is, we have an FPS point with respect to the discrete direction -- is 
\begin{equation*}
\left\{
\begin{split}
\{ |\Delta g_{-1}|^2 (1+s)-\lambda \sigma_{-1} \alpha_{-1}(s) \alpha (s) \} \cdot \{ |\Delta g|^2 (1+s)-\lambda \sigma \alpha(s) \alpha_1 (s) \} >0 , \\
\{ |\Delta g_{-1}|^2 (1-s)+\lambda \sigma_{-1} \alpha_{-1}(s) \alpha (s) \} \cdot \{ |\Delta g|^2 (1-s)+\lambda \sigma \alpha (s) \alpha_1 (s) \} <0 ,
\end{split}
\right.
\end{equation*}
or
\begin{equation*}
\left\{
\begin{split}
\{ |\Delta g_{-1}|^2 (1+s)-\lambda \sigma_{-1} \alpha_{-1}(s) \alpha (s) \} \cdot \{ |\Delta g|^2 (1+s)-\lambda \sigma \alpha (s) \alpha_1 (s) \} <0 , \\
\{ |\Delta g_{-1}|^2 (1-s)+\lambda \sigma_{-1} \alpha_{-1}(s) \alpha (s) \} \cdot \{ |\Delta g|^2 (1-s)+\lambda \sigma \alpha (s) \alpha_1 (s) \} >0 ,
\end{split}
\right.
\end{equation*}
and the condition for $\kappa^x \cdot 
\kappa^x_{1}$ (equivalently, $\kappa^n \cdot 
\kappa^n_{1}$) to be nonpositive -- that is, we have an FPS point with respect to the smooth direction -- is 
\begin{equation*}
\left\{
\begin{split}
\{ |\partial g|(1+s)-\lambda \tau \alpha (s)^2 \} \cdot \{ |\partial g_1|(1+s)-\lambda \tau \alpha_1 (s)^2 \} >0 , \\
\{ |\partial g|(1-s)+\lambda \tau \alpha (s)^2 \} \cdot \{ |\partial g_1|(1-s)+\lambda \tau \alpha_1 (s)^2 \} <0 ,
\end{split}
\right.
\end{equation*}
or
\begin{equation*}
\left\{
\begin{split}
\{ |\partial g|(1+s)-\lambda \tau \alpha (s)^2 \} \cdot \{ |\partial g_1|(1+s)-\lambda \tau \alpha_1 (s)^2 \} <0 , \\
\{ |\partial g|(1-s)+\lambda \tau \alpha (s)^2 \} \cdot \{ |\partial g_1|(1-s)+\lambda \tau \alpha_1 (s)^2 \} >0 ,
\end{split}
\right.
\end{equation*}
\end{theorem}

In \cite{Yashi-next}, the second author established a notion of singular edges for semi-discrete surfaces in Lorentzian spaceforms:

\begin{definition} \label{semis-def55}
An edge $[x,x_1]$ of a semi-discrete 
surface is said to be {\em singular} if 
the tangent plane at this edge 
is not spacelike.  
\end{definition}

With this definition in hand, the second 
author proved this in \cite{Yashi-next}: 
Let $g$ be a semi-discrete holomorphic function and let $x$ be a semi-discrete 
maximal surface determined from $g$ 
by Equation \eqref{semis-eq1}.  
Then an edge $[x,x_1]$ is singular 
if and only if the tangent circle $C$ at $g$, $g_1$ intersects the unit circle $\mathbb{S}^1 = \{ z \in \mathbb{C} \, | \, |z|=1 \}$.

We now prove the following relationship 
between singular points and singular edges: 

\begin{theorem}\label{semis-thm53}
At any singular point $x(k,t)$ 
of a semi-discrete maximal surface with respect to the 
discrete direction such that 
$\kappa_{k-1,k}$ and $\kappa_{k,k+1}$ 
are both 
finite, at least one of the two adjacent 
edges is singular.  

At any singular point with respect to the 
smooth direction, the only possibility is 
that $\kappa$ is infinite there and the 
corresponding image of $g$ lies in 
$\mathbb{S}^1$ there.   
\end{theorem}

\begin{proof}
Proof of the first paragraph: 
\[
0 > \kappa_{-10} \cdot \kappa_{01} = 
\frac{16 |\Delta g_{-10}|^2 
|\Delta g_{01}|^2}{\sigma_{-10} \sigma_{01} 
(1-|g_{-1}|^2) (1-|g|^2)^2 (1-|g_{1}|^2)} \; ,
\]
and thus $0 > (1-|g_{-1}|^2) (1-|g_1|^2)$, which implies exactly one of $g_{-1}$ 
and $g_1$ lies inside $\mathbb{S}^1$. 

Proof of the second paragraph: 
By Equation \eqref{semis-eq2}, $\kappa$ cannot change 
sign, and the result follows.  
\end{proof}

We have the analogous definition for singular 
edges of semi-discrete CMC $1$ 
surfaces in $\mathbb{S}^{2,1}$, as was given 
for semi-discrete maximal surfaces in  $\mathbb{R}^{2,1}$ in 
\cite{Yashi-next}: 

\begin{definition}\label{semis-def56}
Let $n$ be a semi-discrete CMC $1$ surface in $\mathbb{S}^{2,1}$. Then $[n,n_1]$ for some $(k,t) \in \mathbb{Z} \times \mathbb{R}$ is a {\em singular edge} if the plane $\mathcal{P}(n,n_1)$ spanned by $\{ \partial n, \Delta n, \partial \Delta n \}$ is not spacelike. 
\end{definition}

Now, similar to Theorem 1.2 in \cite{Yashi-next}, 
we have the following proposition, which is preparatory for proving 
Theorem \ref{semis-thm55} below.

\begin{theorem}\label{semis-thm54}
Let $n$ be a semi-discrete CMC $1$ surface in $\mathbb{S}^{2,1}$. Then $[n,n_1]$ for some $(k,t) \in \mathbb{Z} \times \mathbb{R}$ is a {\em singular edge} for all $\lambda$ sufficiently close to zero if and only if 
the circle tangent to $\partial g$ and $\partial g_1$ at  
$g$ and $g_1$, respectively, intersects $\mathbb{S}^1$ transversally.  
\end{theorem}

\begin{proof}
By Lemma \ref{semis-lem42}, $\partial (x+n) \parallel \partial n$ and $\Delta (x+n)  \parallel \Delta n$, implying that each tangent plane of $x+n$ at the edge $[x+n,x_1+n_1]$ is parallel to the one for $x$. Thus checking the causality of a tangent plane of $x$ at $[x,x_1]$ is equivalent to checking causality of a tangent plane of $x+n$ at the edge $[x+n,x_1+n_1]$. 

Defining 
$F:=E \begin{pmatrix} 1 & -g \\ 0 & 1 \end{pmatrix}$, we have the following forms:
\begin{eqnarray*}
&&x+n=\frac{F}{\det F} \left( \frac{2}{1-|g|^2} \begin{pmatrix} |g|^2 & g \\ \overline{g} & 1 \end{pmatrix} \right) \overline{F}^t, \\ 
&&x_1+n_1=\frac{F}{\det F} \left( \frac{2}{(1-\lambda \sigma)(1-|g_1|^2)} \begin{pmatrix} |g_1|^2 & g_1 \\ \overline{g_1} & 1 \end{pmatrix} \right) \overline{F}^t \ ,
\end{eqnarray*}
where $F$ and $g$ satisfy
\[
\partial F =F \begin{pmatrix} g&-g^2 \\ 1 &-g\end{pmatrix} \frac{\lambda \tau}{\partial g}, \quad \Delta F =F \begin{pmatrix} g&-g g_1 \\ 1 &-g_1\end{pmatrix} \frac{\lambda \sigma}{\Delta g}, \quad \frac{|\partial g| |\partial g_1|}{|\Delta g|^2}=-\frac{\tau}{\sigma} \ .
\]
Then 
\[
\partial (x+n)=\frac{2}{(1-|g|^2)^2 \det F} F X_1 \bar{F}^t \ , \quad \Delta (x+n)=\frac{2}{(1-\lambda \sigma )(1-|g|^2)(1-|g_1|^2) \det F} F X_2 \bar{F}^t \ , 
\]
where
\begin{eqnarray*}
&&X_1:=\begin{pmatrix} \partial g \cdot \bar{g}+g \cdot \partial \bar{g} & \partial g +g^2 \cdot \partial \bar{g} \\ \partial \bar{g} +\bar{g}^2 \partial g & \partial g \cdot \bar{g}+g \cdot \partial \bar{g} \end{pmatrix}\ , \\
&&X_2:=\begin{pmatrix} |g_1|^2-|g|^2 & \Delta g +g g_1 \Delta \bar{g} \\ \Delta \bar{g} +\bar{g} \overline{g_1} \Delta g & |g_1|^2-|g|^2 \end{pmatrix} +\lambda \sigma (1-|g_1|^2)\begin{pmatrix} |g|^2&g \\ \bar{g} & 1 \end{pmatrix} \ .
\end{eqnarray*}

Because $\partial (x+n)$, $\Delta (x+n)$ and $\partial (x_1+n_1)$ are coplanar, our task is to find a condition, call it condition (C), for the span of $X_1$ and $X_2$ to be non-spacelike for all $\lambda$ close to zero, i.e. $\langle X_1,X_1 \rangle \langle X_2,X_2 \rangle -\langle X_1,X_2 \rangle^2 <0$ (for all $\lambda$ close to 
$0$), and show that this condition (C) is equivalent to the circle tangent to $\partial g$ and $\partial g_1$ at $g$ and $g_1$, respectively, intersecting $\mathbb{S}^1$ transversally. When $|g|=1$, respectively $|g_1|=1$, we know that $\partial (x+n)$, respectively $\partial (x_1+n_1)$, itself is lightlike for all $\lambda \in \mathbb{R} \setminus \{ 0 \}$, so the tangent plane will certainly not be spacelike. Thus, it remains only to find the condition (C) when $|g|$ and $|g_1|$ are both not $1$, and in this case it is
\[
4|\Delta g|^2 |\partial g|^2 (1-|g|^2)(1-|g_1|^2)<\{ (1-\bar{g}g_1)\overline{\Delta g} \cdot \partial g +(1-g\overline{g_1}) \Delta g \cdot \partial \overline{g} \}^2 \ .
\]
A direct computation verifies that this condition (C) is precisely the condition that the circle tangent to $\partial g$ and $\partial g_1$ at  
$g$ and $g_1$, respectively, intersects $\mathbb{S}^1$ transversally.
\end{proof}

We now prove the theorem we have been aiming towards: 

\begin{theorem}\label{semis-thm55}
Consider a point $(k,t)$ in the domain of a semi-discrete 
CMC $1$ surface in $\mathbb{S}^{2,1}$.  Suppose this point is a 
singular point with respect to the 
discrete direction such that 
$\kappa_{k-1,k}$ and $\kappa_{k,k+1}$ 
are both finite, for all $\lambda$ sufficiently close to $0$.   
Then at least one of the two 
adjacent edges is also singular for all $\lambda$ sufficiently close to $0$. 
\end{theorem}

\begin{proof}
From the Weierstrass-type representation, we can write the surface 
as $n$ in $\mathbb{S}^{2,1}$ 
given by some semi-discrete holomorphic function $g$ with 
$s=-1$, with corresponding HMC $1$ surface $x$ in $\mathbb{H}^3$ 
using the same $g$ and same value of $s$. The assumptions regarding 
$(k,t)$ being a singular point for all $\lambda$ close to $0$ imply 
that none of $|g|$, $|g_1|$ and $|g_{-1}|$ are $1$, and also that 
$g_1$ and $g_{-1}$ lie on opposite sides of $\mathbb{S}^1$.  

By Theorem \ref{semis-thm54}, the edge 
$[n(k,t),n(k+1,t)]$ is singular for all $\lambda$ close to zero 
if and only if the circle tangent to $\partial g$ and $\partial g_1$ at  
$g$ and $g_1$, respectively, intersects $\mathbb{S}^1$ transversally.  
From the above properties 
of $g$, $g_1$ and $g_{-1}$, it is clear that at least one of the two edges 
$[n(k,t),n(k+1,t)]$ and $[n(k-1,t),n(k,t)]$ is then singular for all 
$\lambda$ close to zero.  
\end{proof}

\begin{figure}
\begin{center}
\includegraphics[width=150mm]{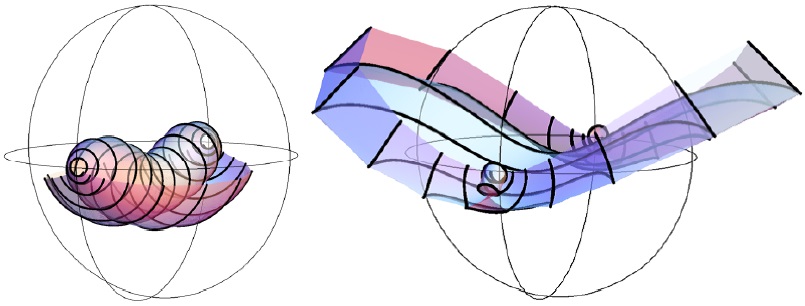} 
\end{center}
\caption{Left-hand side: a semi-discrete CMC $1$ Enneper cousin in $\mathbb{H}^3_+$, right-hand side: a semi-discrete HMC $1$ surface in $\mathbb{H}^3_+ \cup \mathbb{H}^3_-$. The hyperbolic $3$-spaces $\mathbb{H}^3_+$ and $\mathbb{H}^3_-$ are visualized here by stereographically projecting within $4$-dimensional Minkowski space to a horizontal $3$-dimensional spacelike vector subspace.}
\end{figure}

\begin{figure}
\begin{center}
\includegraphics[width=150mm]{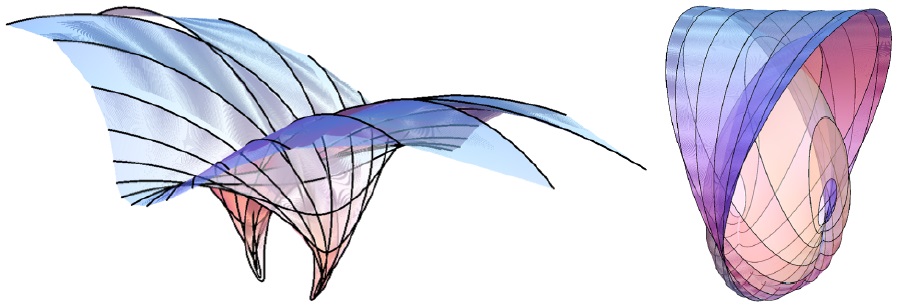}
\end{center}
\caption{Left-hand side: a semi-discrete CMC $1$ Enneper cousin in $\mathbb{S}^{2,1}$, right-hand side: a semi-discrete HMC $1$ surface in $\mathbb{S}^{2,1}$. The de Sitter $3$-space $\mathbb{S}^{2,1}$ is visualized here using the hollow ball model (see \cite{F}).}
\end{figure}

%%%%%%%%%%%%%%%%%%%%%%%%%%%%%%%%%
% References
%%%%%%%%%%%%%%%%%%%%%%%%%%%%%%%%%

\end{document}